\documentclass[]{amsart}
\usepackage{amssymb}
 \newtheorem{Theorem}{Theorem}[section]
 \newtheorem{Corollary}[Theorem]{Corollary}
 \newtheorem{Lemma}[Theorem]{Lemma}

 \newtheorem{Conjecture}[Theorem]{Conjecture}

 \newtheorem{Remark}[Theorem]{Remark}

 \numberwithin{equation}{section}

\begin{document}

\title
 {A weighted version of Saitoh's conjecture}

\author{Qi'an Guan}
\address{Qi'an Guan: School of Mathematical Sciences,
Peking University, Beijing, 100871, China.}
\email{guanqian@math.pku.edu.cn}

\author{Zheng Yuan}

\address{Zheng Yuan: School of Mathematical Sciences,
Peking University, Beijing, 100871, China.}
\email{zyuan@pku.edu.cn}

\thanks{}



\date{\today}

\dedicatory{}

\commby{}

%%% ----------------------------------------------------------------------

\begin{abstract}
In this article, we prove a weighted version of Saitoh's conjecture. As an application, we prove a weighted version of Saitoh's conjecture for higher derivatives. 
\end{abstract}

%%% ----------------------------------------------------------------------
\maketitle
%%% ----------------------------------------------------------------------

\section{Introduction}

Let $D$ be a planar regular region with $n$ boundary components which are analytic Jordan curves (see \cite{saitoh}, \cite{yamada}). 
Let $H^{(c)}_2(D)$ (see \cite{saitoh}) denote the analytic Hardy class on $D$ defined as the set of all analytic functions $f(z)$ on $D$ such that the subharmonic functions $|f(z)|^2$ have harmonic majorants $U(z)$: 
$$|f(z)|^2\le U(z)\,\, \text{on}\,\,D.$$
Then each function $f(z)\in H^{(c)}_2(D)$ has Fatou's nontangential boundary value a.e. on $\partial D$ belonging to $L^2(\partial D)$ (see \cite{duren}). 

Kernel functions associated with various norms have been shown to play a fundamental role in several branches of mathematical analysis (see \cite{Berg70,nehari2}). Let us recall  two reproducing kernels on $D$.

Let $\lambda$ be a positive continuous function on $\partial D$. 
We call  $K_{\lambda}(z,\overline w)$ (see \cite{nehari}) the weighted Szeg\"o kernel if 
$$f(w)=\frac{1}{2\pi}\int_{\partial D}f(z)\overline{K_{\lambda}(z,\overline w)}\lambda(z)|dz|$$
holds for any $f\in H^{(c)}_2(D)$. Let $G_D(p,t)$ be the Green function on $D$, and let $\partial/\partial v_p$ denote the derivative along the outer normal unit vector $v_p$. Fixed $t\in D$, $\frac{\partial G_D(p,t)}{\partial v_p}$ is  positive and continuous on $\partial D$ because of the analyticity of the boundary (see \cite{saitoh}, \cite{guan-19saitoh}). When $\lambda(p)=\left(\frac{\partial G_D(p,t)}{\partial v_p}\right)^{-1}$ on $\partial D$, $\hat K_t(z,\overline w)$ denotes $K_{\lambda}(z,\overline w)$, which is the so-called conjugate Hardy $H^2$ kernel on $D$ (see \cite{saitoh}).
 When $t=w$ and $z=w$, $\hat K(z)$ denotes $\hat K_t(z,\overline w)$ for simplicity.

 Let $\rho$ be a positive Lebesgue measurable function on $D$, which satisfies that there exists $a_{U}>0$ such that $\rho^{-{a_U}}\in L^1(U)$ for any open subset $U\Subset D\backslash Z$, where $Z$ is a discrete subset of $D$. $B_{\rho}(z,\overline w)$ denotes the weighted Bergman kernel on $D$ with the weight $\rho$ (see \cite{pasternak}) if 
 $$f(w)=\int_Df(z)\overline{B_{\rho}(z,\overline w)}\rho(z)$$
 holds for any holomorphic function $f$ on $D$ satisfying $\int_D|f(z)|^2<+\infty$. Denote that
 $$B_{\rho}(z):=B_{\rho}(z,\overline z).$$
 When $\rho\equiv1$, $B(z)$ denotes $B_{\rho}(z)$ for simplicity.

Let $c_\beta(z)$ be the logarithmic capacity which is defined by 
$$c_{\beta}(z):=\exp\lim_{w\rightarrow z}(G_D(w,z)-\log|w-z|).$$
   In \cite{yamada}, Yamada listed the following conjectures on $c_{\beta}(z)$, $B(z)$ and $\hat K(z)$.
 \begin{Conjecture}
 	If $n>1$, then 
 	\begin{equation}
 		\label{eq:0714a}c_{\beta}(z)^2<\pi B(z)<\hat K(z).
 	\end{equation}
 \end{Conjecture}
The left part of inequality \eqref{eq:0714a} is so-called Suita conjecture (see \cite{suita}) and the right part of inequality \eqref{eq:0714a} is so-called Saitoh's conjecture (see \cite{saitoh}).
 
 The original form of Suita conjecture (see \cite{suita}) was posed on open Riemann surfaces admitted nontrivial Green functions. B\l ocki \cite{Blo13} proved the ``$\le $'' part of Suita conjecture  on bounded planar domains. Guan-Zhou \cite{gz12} proved the ``$\le $'' part of Suita conjecture  on open Riemann surfaces. In \cite{guan-zhou13ap}, Guan-Zhou proved a necessary and sufficient condition of the holding of $c_{\beta}(z)^2=\pi B(z)$  on open Riemann surfaces, which completed the proof of Suita conjecture.

In \cite{guan-19saitoh}, Guan proved Saitoh's conjecture:

\begin{Theorem}
	[\cite{guan-19saitoh}]\label{thm:saitoh}
	If $n>1$, then $\hat K(z)>\pi B(z)$.
\end{Theorem}

 We recall some notations (see \cite{OF81}, see also \cite{guan-zhou13ap,GY-concavity,GMY-concavity2}).
 Let $p:\Delta\rightarrow D$ be the universal covering from unit disc $\Delta$ to $ D$, and let $z_0\in D$.
 We call the holomorphic function $f$ on $\Delta$ a multiplicative function,
 if there is a character $\chi$, which is the representation of the fundamental group of $ D$, such that $g^{*}f=\chi(g)f$,
 where $|\chi|=1$ and $g$ is an element of the fundamental group of $ D$. Denote the set of such kinds of $f$ by $\mathcal{O}^{\chi}( D)$.

It is known that for any harmonic function $u$ on $ D$,
there exists a $\chi_{u}$ and a multiplicative function $f_u\in\mathcal{O}^{\chi_u}( D)$,
such that $|f_u|=p^{*}\left(e^{u}\right)$.
If $u_1-u_2=\log|f|$, then $\chi_{u_1}=\chi_{u_2}$,
where $u_1$ and $u_2$ are harmonic functions on $ D$ and $f$ is a holomorphic function on $ D$.
Recall that for the Green function $G_{ D}(z,z_0)$,
there exist a $\chi_{z_0}$ and a multiplicative function $f_{z_0}\in\mathcal{O}^{\chi_{z_0}}( D)$ such that $|f_{z_0}(z)|=p^{*}\left(e^{G_{ D}(z,z_0)}\right)$. $D$ is conformally equivalent to the unit disc (i.e. $n=1$) if and only if $\chi_{z_0}\equiv1$ (see \cite{suita}).

Let $u$ be a harmonic function on $D$, and let $\rho=e^{-2u}$. In \cite{yamada}, Yamada posed the following weighted version of Suita conjecture, which is so-called extended Suita conjecture.
\begin{Conjecture}
	$c_{\beta}^2(z_0)\le \pi\rho(z_0)B_{\rho}(z_0)$, and equality holds if and only if $\chi_{z_0}=\chi_{-u}$.
\end{Conjecture}

In \cite{guan-zhou13ap}, Guan-Zhou proved the extended Suita conjecture.   More general weighted versions  of Suita conjecture can be referred to \cite{GM,GY-concavity}, and a weighted version of Suita conjecture for higher derivatives  can be referred to \cite{GMY-concavity2}.

In the present article, we consider weighted versions of Saitoh's conjecture.

\subsection{Main result}\label{sec:main}

Let $D$ be a planar regular region with $n$ boundary components which are analytic Jordan curves, and let $z_0\in D$. 

Let $\psi$ be a Lebesgue measurable function on  $\overline D$, which satisfies that $\psi$ is  subharmonic on $D$, $\psi|_{\partial D}\equiv 0$
 and the Lelong number $v(dd^c\psi,z_0)>0$, where $d^c=\frac{\partial-\bar\partial}{2\pi\sqrt{-1}}$. Assume that $\psi\in C^1(U\cap\overline{D})$ for an open neighborhood $U$ of $\partial D$ and $\frac{\partial\psi}{\partial v_p}$ is positive on $\partial D$, where $\partial/\partial v_p$ denotes the derivative along the outer normal unit vector $v_p$. Assume that  one of the following two statements holds:
 
  $(a)$ $(\psi-p_0G_{D}(\cdot,z_0))(z_0)>-\infty$, where $p_0=v(dd^c(\psi),z_0)>0$;
  
  $(b)$ $\varphi+2a\psi$ is subharmonic near $z_0$ for some $a\in[0,1)$.

 Let $\varphi$ be a Lebesgue measurable function on $\overline D$ satisfying that $\varphi+2\psi$ is subharmonic on $D$, the Lelong number 
 $$v(dd^c(\varphi+2\psi),z_0)\ge2,$$ and $\varphi$ is continuous at $z$ for any $z\in\partial D$. Let $c$ be a positive Lebesgue measurable function on $[0,+\infty)$ satisfying that $c(t)e^{-t}$ is decreasing on $[0,+\infty)$, $\lim_{t\rightarrow0+0}c(t)=c(0)=1$ and $\int_0^{+\infty}c(t)e^{-t}dt<+\infty$.

 Denote that  
$$\rho:=e^{-\varphi}c(-2\psi)\,\,\text{and}\,\,K_{\rho,\psi}(z):=K_{\rho\left(\frac{\partial \psi}{\partial v_p}\right)^{-1}}(z,\overline z)$$ 
and assume that $\rho$ has a positive lower bound on any compact subset of $D\backslash Z$, where $Z\subset\{\psi=-\infty\}$ is a discrete subset of $D$.

We present a weighted version of Saitoh's conjecture as follows:
\begin{Theorem}
\label{main theorem}Assume that $B_{\rho}(z_0)>0$. Then
	$$K_{\rho,\psi}(z_0)\ge \left(\int_0^{+\infty}c(t)e^{-t}dt\right)\pi B_{\rho}(z_0)$$ holds, and the equality holds if and only if the following statements hold:
	
	$(1)$ $\varphi+2\psi=2G_{D}(\cdot,z_0)+2u$, where $u$ is a harmonic function on $D$;
	 
	$(2)$ $\psi=p_0G_{D}(\cdot,z_0)$, where $p_0=v(dd^c(\psi),z_0)>0$;
	
	$(3)$ $\chi_{z_0}=\chi_{-u}$, where $\chi_{-u}$ and $\chi_{z_0}$ are the  characters associated to the functions $-u$ and $G_{D}(\cdot,z_0)$ respectively.
\end{Theorem} 

\begin{Remark}
	\label{rem:function}
	Let $p$ be the universal covering from unit disc $\Delta$ to $ D$. 
	When the three statements $(1)-(3)$ in  Theorem \ref{main theorem} hold, 
	$$K_{\rho,\psi}(\cdot,\overline{z_0})=\left(\int_0^{+\infty}c(t)e^{-t}dt\right)\pi B_{\rho}(\cdot,\overline{z_0})=c_1(p_*(f_{z_0}))'p_*(f_u),$$
	where $K_{\rho,\psi}(\cdot,\overline{z_0})$ denotes $K_{\rho\left(\frac{\partial\psi}{\partial v_p}\right)^{-1}}(\cdot,\overline{z_0})$, $c_1$ is a constant, $f_u$ is a holomorphic function on $\Delta$ such that $|f_u|=p^*(e^u)$, and $f_{z_0}$ is a holomorphic function on $\Delta$ such that $|f_{z_0}|=p^*(e^{G_D(\cdot,z_0)})$. We prove the remark in Section \ref{sec:proof1}.
\end{Remark}

\begin{Remark}
	 For any $z_0\in D$, there exists  $u\in C(\overline D)$ such that $u$ is harmonic on $D$ and $\chi_{z_0}=\chi_{-u}$. In fact, $u(z):=\log|z-z_0|-G_{D}(z,z_0)$ is harmonic on $D$ and $\chi_{z_0}=\chi_{-u}$.
\end{Remark}

Let $\lambda$ be any positive continuous function on $\partial D$. By solving the Dirichlet problem, there exists $u\in C(\overline D)$ satisfying that $u|_{\partial D}=-\frac{1}{2}\log\lambda$ and $u$ is harmonic on $D$. When $\psi=G_{D}(\cdot,z_0)$, $\hat K_{\lambda}(z_0)$ denotes $K_{\lambda,\psi}(z_0)$.

Theorem \ref{main theorem} implies the following corollary.
\begin{Corollary}
	$\hat K_{\lambda}(z_0)\ge \pi B_{e^{-2u}}(z_0)$ holds for any $z_0\in D$, and the equality holds if and only if $\chi_{z_0}=\chi_{-u}$.
\end{Corollary}

Note that $\chi_{z_0}\equiv1$ holds if and only if $n=1$ (see \cite{suita}), then 
 the above corollary is Theorem \ref{thm:saitoh} when $\lambda\equiv1$ and $u\equiv0$.

\subsection{Applications: the weighted version of Saitoh's conjecture for higher derivatives }
Let $D$ be a planar regular region with $n$ boundary components which are analytic Jordan curves, and let $z_0\in D$.  

Let $\psi$ be a Lebesgue measurable function on  $\overline D$, which satisfies that $\psi$ is  subharmonic on $D$, $\psi|_{\partial D}\equiv 0$
 and the Lelong number $v(\psi,z_0)>0$. Assume that $\psi\in C^1(U\cap\overline{D})$ for an open neighborhood $U$ of $\partial D$ and $\frac{\partial\psi}{\partial v_p}$ is positive on $\partial D$. Assume that  one of the following two statements holds:
 
  $(a)$ $(\psi-p_0G_{D}(\cdot,z_0))(z_0)>-\infty$, where $p_0=v(dd^c(\psi),z_0)>0$;
  
  $(b)$ $\varphi+2a\psi$ is subharmonic near $z_0$ for some $a\in[0,1)$.
  
Let $k$ be a nonnegative integer. Let $\varphi$ be a Lebesgue measurable function on $\overline D$ satisfying that $\varphi+2\psi$ is subharmonic on $D$, the Lelong number $$v(dd^c(\varphi+2\psi),z_0)\ge2(k+1),$$
 and $\varphi$ is continuous at $z$ for any $z\in\partial D$. Let $c$ be a positive Lebesgue measurable function on $[0,+\infty)$ satisfying that $c(t)e^{-t}$ is decreasing on $[0,+\infty)$, $\lim_{t\rightarrow0+0}c(t)=c(0)=1$ and $\int_0^{+\infty}c(t)e^{-t}dt<+\infty$.

  Denote that
$$\rho:=e^{-\varphi}c(-2\psi),$$ 
and assume that $\rho$ has a positive lower bound on any compact subset of $D\backslash Z$, where $Z\subset\{\psi=-\infty\}$ is a discrete subset of $D$. 

Let us consider two kernel functions for higher derivatives.
Denote that
\begin{displaymath}
	\begin{split}
		&B^{(k)}_{\rho}(z_0)\\
		&:=\sup\left\{\left|\frac{f^{(k)}(z_0)}{k!}\right|^2:f\in\mathcal{O}(D),\,\,\int_D|f|^2\rho\le1\,\,\&\,\,f(z_0)=\ldots=f^{(k-1)}(z_0)=0 \right\}.
	\end{split}
\end{displaymath}
When $\rho\equiv1$, $B^{(k)}_{\rho}(z_0)$ is the Bergman kernel for higher derivatives (see \cite{Berg70,Blo18}). When $k=0$, $B^{(k)}_{\rho}(z_0)$ is the weighted Bergman kernel $B_{\rho}(z_0)$ (see Section \ref{sec:main}).
Denote that 
\begin{displaymath}
	\begin{split}
		K_{\rho,\psi}^{(k)}(z_0):=&\sup\Bigg\{\left|\frac{f^{(k)}(z_0)}{k!}\right|^2:f\in H^{(c)}_2(D),\\
		&\int_{\partial D}|f|^2\rho\left(\frac{\partial\psi}{\partial v_z} \right)^{-1}|dz|\le1\,\,\&\,\,f(z_0)=\ldots=f^{(k-1)}(z_0)=0 \Bigg\}.
	\end{split}
\end{displaymath}
Especially, when $k=0$, $K_{\rho,\psi}^{(k)}(z_0)$ is the weighted  Szeg\"o kernel $K_{\rho,\psi}(z_0)$ (see Section \ref{sec:main}).

We present a weighted version of Saitoh's conjecture for higher derivatives as follows:

\begin{Corollary}
	\label{c:higher}Assume that $B^{(k)}_{\rho}(z_0)>0$. Then
	$$K_{\rho}^{(k)}(z_0)\ge\left(\int_0^{+\infty}c(t)e^{-t}dt\right) \pi B^{(k)}_{\rho}(z_0)$$
	 holds, and the equality holds if and only if  the following statements hold:
	
	$(1)$ $\varphi+2\psi=2(k+1)G_{D}(\cdot,z_0)+2u$, where $u$ is a harmonic function on $D$;
	 
	$(2)$ $\psi=p_0G_{D}(\cdot,z_0)$, where $p_0=v(dd^c(\psi),z_0)>0$;
	
	$(3)$ $\chi_{z_0}^{k+1}=\chi_{-u}$, where $\chi_{-u}$ and $\chi_{z_0}$ are the  characters associated to the functions $-u$ and $G_{D}(\cdot,z_0)$ respectively.
\end{Corollary}

Let $\lambda$ be arbitrary positive continuous function on $\partial D$. By solving the Dirichlet problem, there exists $u\in C(\overline D)$ satisfying that $u|_{\partial D}=-\frac{1}{2}\log\lambda$ and $u$ is harmonic on $D$. When $\psi=(k+1)G_{D}(\cdot,z_0)$, $\hat K^{(k)}_{\lambda}(z_0)$ denotes $K^{(k)}_{\lambda,\psi}(z_0)$.

Corollary \ref{c:higher} implies the following corollary.
\begin{Corollary}
	$\hat K^{(k)}_{\lambda}(z_0)\ge \pi B^{(k)}_{e^{-2u}}(z_0)$ holds for any $z_0\in D$, and the equality holds if and only if $\chi_{z_0}^{k+1}=\chi_{-u}$.
\end{Corollary}

\section{Preparations}

In this section, we do some preparations.

\subsection{A sufficient condition for $f\in H^{(c)}_2(D)$}

Let $D$ be a planar regular region with $n$ boundary components which are analytic Jordan curves, and let $z_0\in D$. Let $\psi$ be as in Theorem \ref{main theorem}. Let $f$ be a holomorphic function on $D$. In this section, we give a sufficient condition for $f\in H^{(c)}_2(D)$ (i.e. Lemma \ref{l:2}).

We recall the following basic formula, and we give a proof for the convenience of readers.
\begin{Lemma}
	\label{l:4}$\frac{\partial \psi}{\partial v_z}=\left(\left(\frac{\partial \psi}{\partial x}\right)^2+\left(\frac{\partial \psi}{\partial y}\right)^2\right)^{\frac{1}{2}}$ on $\partial D$, where $\partial/\partial v_z$ denotes the derivative along the outer normal unit vector $v_z$. 
\end{Lemma}
\begin{proof}
		Fixed $z_1\in\partial D$, as $\frac{\partial \psi}{\partial v_z}$ is positive on $D$, we can assume that $\frac{\partial\psi}{\partial y}\not=0$ without loss of generality. Then there exists a neighborhood $U_1$ of $z_1$ with coordinates $(u,v)=(x,\psi(x+\sqrt{-1}y))$. It is clear that 
	\begin{displaymath}
		\frac{\partial u}{\partial x}=1,\,\,\frac{\partial u}{\partial y}=0,\,\,\frac{\partial v}{\partial x}=\frac{\partial \psi}{\partial x}\,\,\text{and}\,\,\frac{\partial v}{\partial y}=\frac{\partial \psi}{\partial y},
	\end{displaymath}
	which implies that 
	\begin{displaymath}
		\frac{\partial x}{\partial u}=1,\,\,\frac{\partial y}{\partial u}=-\frac{\frac{\partial \psi}{\partial x}}{\frac{\partial \psi}{\partial y}} ,\,\,\frac{\partial x}{\partial v}=0\,\,\text{and}\,\,\frac{\partial y}{\partial v}=\left(\frac{\partial \psi}{\partial y}\right)^{-1}.
	\end{displaymath}
It is clear that 
$$v_z=\frac{(\frac{\partial\psi}{\partial x},\frac{\partial\psi}{\partial y} )}{\left((\frac{\partial\psi}{\partial x})^2+(\frac{\partial\psi}{\partial y})^2\right)^{\frac{1}{2}}},$$
thus we have $\frac{\partial \psi}{\partial v_z}=\left(\left(\frac{\partial \psi}{\partial x}\right)^2+\left(\frac{\partial \psi}{\partial y}\right)^2\right)^{\frac{1}{2}}$.
\end{proof}

We give a relationship between the superlevel sets of $\psi$ and $G_D(\cdot,z_0)$.
\begin{Lemma}
	\label{l:1}There exist $t_0>0$ and $C>1$ such that 
	$$\{z\in D:G_D(z,z_0)\ge-t\}\subset\{z\in D:\psi(z)\ge-Ct\}$$
	for any $t\in(0,t_0)$.
\end{Lemma}
\begin{proof}
	As $\partial D$ is compact, it suffices to prove that for any $z_1\in\partial D$, there exist  a neighborhood $U$ of $z_1$, $t_0>0$ and $C>1$ such that $\{z\in D\cap U:G_D(z,z_0)\ge-t\}\subset\{z\in D\cap U:\psi(z)\ge-Ct\}$ for any $t\in(0,t_0)$.
	
	Fixed $z_1\in\partial D$, as $\frac{\partial G_{D}(z,z_0)}{\partial v_z}$ is positive on $D$, we can assume that $\frac{\partial G_{D}(z,z_0)}{\partial y}\not=0$ and $z_1$ is the origin $o$ in $\mathbb{C}$ without loss of generality. Then there exists a neighborhood $U_1$ of $z_1$ with coordinates $(u,v)=(x,G_D(x+\sqrt{-1}y,z_0))$. It is clear that 
	\begin{displaymath}
		\frac{\partial u}{\partial x}=1,\,\,\frac{\partial u}{\partial y}=0,\,\,\frac{\partial v}{\partial x}=\frac{\partial}{\partial x}G(z,z_0)\,\,\text{and}\,\,\frac{\partial v}{\partial y}=\frac{\partial}{\partial y}G_D(z,z_0),
	\end{displaymath}
	which implies that 
	\begin{displaymath}
		\frac{\partial x}{\partial u}=1,\,\,\frac{\partial y}{\partial u}=-\frac{\frac{\partial}{\partial x}G(z,z_0)}{\frac{\partial}{\partial y}G_D(z,z_0)},\,\,\frac{\partial x}{\partial v}=0\,\,\text{and}\,\,\frac{\partial y}{\partial v}=\left(\frac{\partial}{\partial y}G_D(z,z_0)\right)^{-1}.
	\end{displaymath}
	It is clear that 
$$v_z=\frac{(\frac{\partial G_D(z,z_0)}{\partial x},\frac{\partial G_D(z,z_0)}{\partial y} )}{\left((\frac{\partial G_D(z,z_0)}{\partial x})^2+(\frac{\partial G_D(z,z_0)}{\partial y})^2\right)^{\frac{1}{2}}}$$
on $\partial D$.
	Thus,  we have
	\begin{displaymath}
		\begin{split}
			&\frac{\partial \psi}{\partial u}\cdot\frac{\partial G_D(z,z_0)}{\partial x}+\frac{\partial \psi}{\partial v}\cdot|\bigtriangledown G_D(z,z_0)|^2\\
			=&\left(\frac{\partial\psi}{\partial x}\cdot\frac{\partial x}{\partial u}+\frac{\partial\psi}{\partial y}\cdot\frac{\partial y}{\partial u} \right)\frac{\partial G_D(z,z_0)}{\partial x}+\left(\frac{\partial\psi}{\partial x}\cdot\frac{\partial x}{\partial v}+\frac{\partial\psi}{\partial y}\cdot\frac{\partial y}{\partial v} \right)|\bigtriangledown G_D(z,z_0)|^2\\
			=&\left(\frac{\partial\psi}{\partial x}-\frac{\partial\psi}{\partial y}\cdot\frac{\frac{\partial}{\partial x}G(z,z_0)}{\frac{\partial}{\partial y}G_D(z,z_0)} \right)\frac{\partial G_D(z,z_0)}{\partial x}\\
			&+\frac{\partial\psi}{\partial y}\cdot\left(\frac{\partial}{\partial y}G_D(z,z_0)\right)^{-1}\cdot\left(\left(\frac{\partial G_D(z,z_0)}{\partial x}\right)^{2}+\left(\frac{\partial G_D(z,z_0)}{\partial y}\right)^{2}\right)\\
			=&\frac{\partial\psi}{\partial y}\cdot \frac{\partial G_D(z,z_0)}{\partial y}+\frac{\partial\psi}{\partial x}\cdot \frac{\partial G_D(z,z_0)}{\partial x}\\
			=&\frac{\frac{\partial \psi}{\partial v_z}}{\left((\frac{\partial G_D(z,z_0)}{\partial x})^2+(\frac{\partial G_D(z,z_0)}{\partial y})^2\right)^{\frac{1}{2}}}>0
		\end{split}
	\end{displaymath}
	on $\partial D$. Note that $|\bigtriangledown G_D(z,z_0)|^2>0$ on $\partial D$. There exist $a\in \mathbb{R}$, $m>0$, $r_0>0$ and $b>0$ such that 
	\begin{equation}
		\label{eq:0710a}m<a\frac{\partial \psi}{\partial u}+b\frac{\partial \psi}{\partial v}<\frac{1}{m}	\end{equation}
	on an open parallelogram $U_2:=\{(u,v):-r_0<v<r_0,\,\,\frac{a}{b}v-r_0<u<\frac{a}{b}v+r_0 \}\Subset U_1$. Note that $\psi|_{\{v=0\}}=\psi|_{\partial D}\equiv0$. For any $(u,v)\in U_2$, we have $(u-\frac{a}{b}v+ta,tb)\in U_2$ for any $t\in[0,\frac{v}{b}]$ and
	\begin{equation}
		\label{eq:0710b}\begin{split}
			\psi(u,v)&=\psi(u,v)-\psi(u-\frac{a}{b}v,0)\\
			&=\psi(u-\frac{a}{b}v+ta,tb)\bigg|_{t=0}^{t=\frac{v}{b}}\\
			&=\int_0^{\frac{v}{b}}\left(a\frac{\partial\psi}{\partial u}+b\frac{\partial\psi}{\partial v} \right)(u-\frac{a}{b}v+ta,tb)dt.
		\end{split}
	\end{equation}
	Thus, for any $t\in(0,r_0)$, if $G(z,z_0)=v\ge-t$, it follows form inequality \eqref{eq:0710a} and equality \eqref{eq:0710b} that 
	\begin{displaymath}
		\begin{split}
			\psi(u,v)&=-\int_{\frac{v}{b}}^0\left(a\frac{\partial\psi}{\partial u}+b\frac{\partial\psi}{\partial v} \right)(u-\frac{a}{b}v+ta,tb)dt\\
			&\geq\frac{v}{mb}\\
			&\geq -\frac{t}{mb},
		\end{split}
	\end{displaymath} 
	which implies that $\{z\in D\cap U_2:G_D(z,z_0)\ge-t\}\subset\{z\in D\cap U_2:\psi(z)\ge-\frac{1}{mb}t\}$ for any $t\in(0,r_0)$.
	
	Thus, Lemma \ref{l:1} holds.
\end{proof}

We recall the following coarea formula.
\begin{Lemma}[see \cite{federer}]
	\label{l:3}Suppose that $\Omega$ is an open set in $\mathbb{R}^n$ and $u\in C^1(\Omega)$. Then for any $g\in L^1(\Omega)$, 
	$$\int_{\Omega}g(x)|\bigtriangledown u(x)|dx=\int_{\mathbb{R}}\left(\int_{u^{-1}(t)}g(x)dH_{n-1}(x)\right)dt,$$
	where $H_{n-1}$ is the $(n-1)$-dimensional Hausdorff measure.
\end{Lemma}

The following lemma give a sufficient condition for $f\in H^{(c)}_2(D)$.

\begin{Lemma}
	\label{l:2}
	Let $f$ be a holomorphic function on $D$.   Assume that 
	\begin{equation}
		\label{eq:0708a}\liminf_{r\rightarrow1-0}\frac{\int_{\{z\in D:\psi(z)\ge\log r\}}|f(z)|^2}{1-r}<+\infty,
	\end{equation}
	then we have $f\in H^{(c)}_2(D)$.
\end{Lemma}
\begin{proof}
	It follows from Lemma \ref{l:1} and inequality \eqref{eq:0708a} that 
	\begin{equation}
		\label{eq:0708b}\begin{split}
					&\liminf_{r\rightarrow1-0}\frac{\int_{\{z\in D:e^{G_D(z,z_0)}\ge r\}}|f(z)|^2}{1-r}\\
					\le&		\liminf_{r\rightarrow1-0}\frac{\int_{\{z\in D:\psi(z)\ge C\log r\}}|f(z)|^2}{1-r}\\
					=&\liminf_{r\rightarrow1-0}\frac{\int_{\{z\in D:\psi(z)\ge C\log r\}}|f(z)|^2}{1-r^{C}}\times\frac{1-r^{C}}{1-r}\\
					<&+\infty.
		\end{split}
	\end{equation}
	Denote that $$D_r:=\{z\in D:e^{G_D(z,z_0)}<r\},$$
	 where $r\in(0,1)$. It is well-known that $G_D(\cdot,z_0)-\log r$ is the Green function on $D_r$.
	By the analyticity of the boundary of $D$, we have $G_D(z,w)$ has an analytic extension on $U\times V\backslash\{z=w\}$ and $\frac{\partial G_D(z,z_0)}{\partial v_z}$ is positive and smooth on $\partial D$, where $U$ is an neighborhood of $\overline{D}$ and $V\Subset D$. Then there exist $r_0\in(0,1)$ and $C_1>0$ such that  $\frac{1}{C_1}\le|\bigtriangledown G_D(\cdot,z_0)|\le C_1$ on $\{z\in D:G_D(z,z_0)>\log r_0\}$, which implies 
	\begin{equation}
		\label{eq:0709b}\frac{1}{C_1}\le\frac{\partial G_{D}(z,z_0)}{\partial v_z}\le C_1
	\end{equation}
	holds on $\{z\in D:G_D(z,z_0)>\log r_0\}$ (by using Lemma \ref{l:4}).

	Denote that $$v_r(w):=\frac{1}{2\pi}\int_{\partial D_r}|f|^2\frac{\partial G_{D_r}(z,w)}{\partial v_z}|dz|$$
is a harmonic function on $D_r$, where $r\in(r_0,1)$.  As $G_{D_r}(z,z_0)=G_{D}(z,z_0)-\log r$, we have 
\begin{equation}
	\label{eq:0708c}v_r(z_0)=\frac{1}{2\pi}\int_{\partial D_r}|f|^2\frac{\partial G_{D}(z,z_0)}{\partial v_z}|dz|.
\end{equation}
	Fixed $r_1\in(r_0,1)$,  inequality \eqref{eq:0709b} implies that 
	\begin{equation}
		\label{eq:0709a}\begin{split}
			v_{r_1}(z_0)&\le v_r(z_0)\\
			&= \frac{1}{2\pi}\int_{\partial D_r}|f|^2\frac{\partial G_{D}(z,z_0)}{\partial v_z}|dz|\\
			&\le C_2 \int_{\partial D_r}|f|^2\left(\frac{\partial G_{D}(z,z_0)}{\partial v_z}\right)^{-1}|dz|
		\end{split}
	\end{equation}
holds	for any $r\in(r_1,1)$, where $C_2$ is a positive constant independent of $r_1$ and $r$. Using Lemma \ref{l:4}, Lemma \ref{l:3} and inequality \eqref{eq:0708b}, we have
\begin{equation}
	\label{eq:0709c}
	\begin{split}
		v_{r_1}(z_0)&\le\liminf_{r\rightarrow1-0}\frac{C_2 \int_{r}^1\left(\int_{\partial D_s}|f|^2\left(\frac{\partial G_{D}(z,z_0)}{\partial v_z}\right)^{-1}|dz|\right)ds}{1-r}\\
		&=\liminf_{r\rightarrow1-0}\frac{C_2 \int_{r}^1\left(\int_{\{e^{G_{D}(\cdot,z_0)}=s\}}|f|^2e^{G_{D}(z,z_0)}|\bigtriangledown e^{G_{D}(z,z_0)}|^{-1}|dz|\right)ds}{1-r}\\
		&=\liminf_{r\rightarrow1-0}\frac{C_2 \int_{\{z\in D:e^{G_{D}(z,z_0)}>r\}}|f|^2e^{G_{D}(z,z_0)}}{1-r}\\
		&\le C_3,
	\end{split}
\end{equation}
	where $C_3$ is a positive constant independent of $r_1$. As $|f|^2$ is subharmonic, we have  $|f|^2\le v_r$ on $D_r$ and $\{v_r\}$ is increasing with respect to $r$. By Harnack's principle (see \cite{ahlfors}), the sequence $\{v_r\}$ converges to a harmonic function $v$ on $D$, which satisfies that $|f(z)|^2\le v(z)$ for any $z\in D$. Thus, $f\in H^{(c)}_2(D)$.
	\end{proof}
	
	\subsection{Concavity property of minimal $L^2$ integrals} In this section, we recall the concavity property of minimal $L^2$ integrals on open Riemann surfaces and a characterization for the concavity degenerating to linearity (\cite{GY-concavity}, see also \cite{GMY-concavity2,GY-concavity3}).
	
	Let $D$ be a planar regular region with $n$ boundary components which are analytic Jordan curves.
	Let $\psi$ be a negative subharmonic function on $D$,
and let  $\varphi$ be a Lebesgue measurable function on $D$,
such that $\varphi+\psi$ is a plurisubharmonic function on $D$.

Let $z_0\in D$  such that $\mathcal{I}(\varphi+\psi)_{z_0}\not=\mathcal{O}_{z_0}$, where $\mathcal{I}(\varphi+\psi)$ is the multiplier ideal sheaf, which is the sheaf of germs of holomorphic functions $h$ such that $|h|^2e^{-\varphi-\psi}$ is locally integrable.
Let $f$ be a holomorphic function on a neighborhood of $z_0$.
Let $\mathcal{F}_{z_0}\supseteq\mathcal{I}(\varphi+\psi)_{z_0}$ be an  ideal of $\mathcal{O}_{z_0}$.

Denote
\begin{equation*}
\begin{split}
\inf\bigg\{\int_{\{\psi<-t\}}|\tilde{f}|^{2}e^{-\varphi}c(-\psi):(\tilde{f}-f,z_0)\in \mathcal{F}_{z_0}\,\&{\,}\tilde{f}\in \mathcal{O} (\{\psi<-t\})\bigg\},
\end{split}
\end{equation*}
by $G(t;c)$ (without misunderstanding, we denote $G(t;c)$ by $G(t)$),  where $t\in[0,+\infty)$ and  $c$ is a nonnegative measurable function on $(0,+\infty)$.

Let $c$ be
 a positive measurable function $c$  on $(0,+\infty)$, which satisfies that 
$c(t)e^{-t}$ is decreasing with respect to $t$, $\int_0^{+\infty}c(s)e^{-s}ds<+\infty$ and $e^{-\varphi}c(-\psi)$ has a positive lower bound on any compact subset of $D\backslash Z$, where $Z\subset\{\psi=-\infty\}$ is a discrete subset of $M$.	

 We recall some results about the concavity for $G(t)$, which will be used in the proof of Theorem \ref{main theorem}.

\begin{Theorem}[\cite{GY-concavity}]
\label{thm:general_concave}
 $G(h^{-1}(r))$ is concave with respect to $r\in(0,\int_{0}^{+\infty}c(s)e^{-s}ds)$, $\lim_{t\rightarrow T+0}G(t)=G(0)$ and $\lim_{t\rightarrow +\infty}G(t)=0$, where $h(t)=\int_{t}^{+\infty}c(s)e^{-s}ds$.
\end{Theorem}

\begin{Lemma}
	[\cite{GY-concavity}]\label{l:unique}  There exists a unique holomorphic function $F$ on $\{\psi<-t\}$ satisfying $(F-f,z_0)\in \mathcal F_{z_0}$ and $G(t;c)=\int_{\{\psi<-t\}}|F|^2e^{-\varphi}c(-\psi)$. Furthermore,
for any holomorphic function $\hat{F}$ on $\{\psi<-t\}$ satisfying $(\hat{F}-f,z_0)\in \mathcal{F}_{z_0}$ and
$\int_{\{\psi<-t\}}|\hat{F}|^{2}e^{-\varphi}c(-\psi)<+\infty$,
we have the following equality
\begin{equation*}
\begin{split}
&\int_{\{\psi<-t\}}|F_{t}|^{2}e^{-\varphi}c(-\psi)+\int_{\{\psi<-t\}}|\hat{F}-F_{t}|^{2}e^{-\varphi}c(-\psi)
\\=&
\int_{\{\psi<-t\}}|\hat{F}|^{2}e^{-\varphi}c(-\psi).
\end{split}
\end{equation*}
\end{Lemma}

We recall a necessary condition and a characterization of the concavity degenerating to linearity.

\begin{Corollary}[\cite{GY-concavity}]	\label{c:linear}
 If $G( {h}^{-1}(r))$ is linear with respect to $r\in[0,\int_{0}^{+\infty}c(s)e^{-s}ds)$, where $ {h}(t)=\int_{t}^{+\infty}c(s)e^{-s}ds$,
 then there is a unique holomorphic function $F$ on $D$ satisfying that $(F-f,z_0)\in \mathcal F_{z_0}$ and $G(t;c)=\int_{\{\psi<-t\}}|F|^2e^{-\varphi}c(-\psi)$ for any $t\geq 0$. Furthermore,
\begin{equation}
	\label{eq:20210412b}
	\int_{\{-t_1\leq\psi<-t_2\}}|F|^2e^{-\varphi}a(-\psi)=\frac{G(0;c)}{\int_{0}^{+\infty}c(s)e^{-s}ds}\int_{t_2}^{t_1} a(t)e^{-t}dt
\end{equation}
for any nonnegative measurable function $a$ on $(0,+\infty)$, where $+\infty\geq t_1>t_2\geq 0$.
\end{Corollary}

\begin{Theorem}[\cite{GY-concavity}, see also \cite{GY-concavity3}]
	\label{thm:m-points}
 Assume that  one of the following two statements holds:
 
  $(a)$ $(\psi-2p_0G_{D}(\cdot,z_0))(z_0)>-\infty$, where $p_0=\frac{1}{2}v(dd^c(\psi),z_0)>0$;
  
  $(b)$ $\varphi+a\psi$ is subharmonic near $z_0$ for some $a\in[0,1)$.
  
   Then $G(h^{-1}(r))$ is linear with respect to $r$ if and only if the following statements hold:
	
	$(1)$ $\psi=2p_0G_{D}(\cdot,z_0)$, where $p_0=\frac{1}{2}v(dd^c(\psi),z_0)>0$;
	
	$(2)$ $\varphi+\psi=2\log|g|+2G_{D}(\cdot,z_0)+2u$ and $\mathcal{F}_{z_0}=\mathcal{I}(\varphi+\psi)_{z_0}$, where $g$ is a holomorphic function on $D$ such that $ord_{z_0}(g)=ord_{z_0}(f)$ and $u$ is a harmonic function on $D$;
	
	$(3)$ $\chi_{z_0}=\chi_{-u}$, where $\chi_{-u}$ and $\chi_{z_0}$ are the  characters associated to the functions $-u$ and $G_{D}(\cdot,z_0)$ respectively.
\end{Theorem}

\begin{Remark}[\cite{GY-concavity3}]\label{rem:1.1}
Assume the three statements $(1)-(3)$ in Theorem \ref{thm:m-points} hold. Let $p$ be the universal covering from unit disc $\Delta$ to $ D$. Let $f_u$ be a holomorphic function on $\Delta$ such that $|f_u|=p^*(e^u)$, and let $f_{z_0}$ be a holomorphic function on $\Delta$ such that $|f_{z_0}|=p^*(e^{G_D(\cdot,z_0)})$. Denote that $c_0:=\lim_{z\rightarrow z_0}\frac{f}{p_0gp_*(f_u)(p_*(f_{z_0}))'}$. Then 
$$c_0p_0gp_*(f_u)(p_*(f_{z_0}))'$$
 is the unique holomorphic function $F$ on $D$ such that $(F-f,z_0)\in\mathcal{F}_{z_0}$  and
	$G(t)=\int_{\{\psi<-t\}}|F|^2e^{-\varphi}c(-\psi)$ for any $t\ge0$.
\end{Remark}

\subsection{Some other required results}
Let $D$ be a planar regular region with $n$ boundary components which are analytic Jordan curves, and let $z_0\in D$.	
	
\begin{Lemma}
	[see \cite{S-O69}, see also \cite{Tsuji}]\label{l:green-sup}$G_{D}(z,z_0)=\sup_{v\in\Delta_{D}^*(z_0)}v(z)$, where $\Delta_{D}^*(z_0)$ is the set of negative subharmonic function on $D$ such that $v(z)-\log|z-z_0|$ has a locally finite upper bound near $z_0$. Moreover, $G_{D}(z,z_0)-\log|z-z_0|$ is harmonic on $D$.
\end{Lemma}	

The following two properties of the weighted Szeg\"o kernel can be referred to \cite{nehari}.

\begin{Lemma}[\cite{nehari}]
Let $\lambda$ be a positive continuous function on $\partial D$.
	\label{l:szego1}There exists an analytic function $K_{\lambda}(z,\overline w)$ with the following properties: $K_{\lambda}(z,\overline w)$ is holomorphic on $D\times D$; $|K_{\lambda}(z,\overline w)|$ is continuous on $\overline{D}$ for fixed $w\in D$;
	$$\int_{\partial D}f(z)\overline{K_{\lambda}(z,\overline w)}\lambda(z)|dz|=f(w)$$
	holds for any $f\in H^{(c)}_2(D)$.
\end{Lemma}
	
	\begin{Lemma}[\cite{nehari}]\label{l:szego2}
		Let $\lambda$ be a positive continuous function on $\partial D$, and let $f\in H^{(c)}_2(D)$ satisfying $f(z_0)=1$. Then we have
		\begin{equation}\label{eq:220710b}
			\int_{\partial D}|M(z)|^2\lambda(z)|dz|\le \int_{\partial D}|f(z)|^2\lambda(z)|dz|,
		\end{equation}
		where $M(z):=\frac{K_{\lambda}(z,\overline{z_0})}{K_{\lambda}(z_0,\overline{z_0})}$.
		Equality in \eqref{eq:220710b} holds if and only if $f(z)\equiv M(z)$.
	\end{Lemma}

\section{Proofs of Theorem \ref{main theorem} and Remark \ref{rem:function}}\label{sec:proof1}
In this section, we prove 
Theorem \ref{main theorem} and Remark \ref{rem:function}.

\begin{proof}[Proof of Theorem \ref{main theorem}]

We prove Theorem \ref{main theorem} in three steps: Firstly we prove that ``$\geq$" holds; secondly we prove the necessity of the characterization; finally we prove the sufficiency of the characterization.

\

\emph{Step 1:} 
Denote \begin{equation*}
\begin{split}
\inf\bigg\{\int_{\{2\psi<-t\}}|\tilde{f}|^{2}e^{-\varphi}c(-2\psi):\tilde{f}(z_0)=1\,\&{\,}\tilde{f}\in \mathcal{O} (\{2\psi<-t\})\bigg\},
\end{split}
\end{equation*}
by $G(t)$ for $t\ge0$, then we have 
$$G(0)=\frac{1}{B_{\rho}(z_0)},$$ where $\rho=e^{-\varphi}c(-2\psi)$.
 Lemma \ref{l:unique} tells us that there exists a holomorphic function $F_0$ on $D$ such that $F_0(z_0)=1$ and $G(0)=\int_{D}|F_0|^2e^{-\varphi}c(-2\psi)$. Theorem \ref{thm:general_concave} shows that $G(h^{-1}(r))$ is concave, where $h(t)=\int_t^{+\infty}c(s)e^{-s}ds$. Note that 
 $$G(-\log r)\le \int_{\{2\psi<\log r\}}|F_0|^2e^{-\varphi}c(-2\psi)$$
 for $r\in(0,1]$,
then we have 
\begin{equation}
	\label{eq:0709d}\frac{\int_{\{z\in D:2\psi(z)\ge\log r\}}|F_0(z)|^2e^{-\varphi}c(-2\psi)}{\int_0^{-\log r}c(t)e^{-t}dt}\le \frac{G(0)-G(-\log r)}{\int_0^{-\log r}c(t)e^{-t}dt}\le \frac{G(0)}{\int_0^{+\infty}c(t)e^{-t}dt}.
\end{equation}
There exists $r_0\in(0,1)$ such that $\inf\{e^{-\varphi(z)}c(-\psi(z)):z\in D\,\&\,2G_{D}(z,z_0)\ge\log r_0\}>0$.
As $v(dd^c\psi,z_0)>0$, it follows from Lemma \ref{l:green-sup} that there exists $r_1\in(0,1)$ such that $\{z\in D:2\psi(z)\ge\log r_1\}\subset\{z\in D:2G_D(z,z_0)\ge\log r_0\}$. Note that $\lim_{t\rightarrow0+0}c(t)=1$. Then inequality \eqref{eq:0709d} implies that 
\begin{displaymath}
	\begin{split}
		&\liminf_{r\rightarrow1-0}\frac{\int_{\{z\in D:2\psi(z)\ge\log r\}}|F_0(z)|^2}{1-r}\\
		\le& C_1\liminf_{r\rightarrow1-0}\frac{\int_{\{z\in D:2\psi(z)\ge\log r\}}|F_0(z)|^2e^{-\varphi}c(-2\psi)}{\int_0^{-\log r}c(t)e^{-t}dt}\times\frac{\int_0^{-\log r}c(t)e^{-t}dt}{1-r}\\
		\le &C_1\frac{G(0)}{\int_0^{+\infty}c(t)e^{-t}dt}\liminf_{r\rightarrow1-0}\frac{\int_0^{-\log r}c(t)e^{-t}dt}{1-r}\\
		<&+\infty.
	\end{split}
\end{displaymath}
Using Lemma \ref{l:2}, we have $F_0\in H_2^{(c)}(D)$. 

Note that $F_0$ has Fatou's nontangential boundary value and  $|F_0|\in L^2(\partial D)$.
It follows from Fatou's Lemma, Lemma \ref{l:4} and Lemma \ref{l:3} that 
\begin{equation}
	\label{eq:0709e}\begin{split}
		&\int_{\partial D}|F_0|^2e^{-\varphi}c(-2\psi)\left(\frac{\partial\psi}{\partial v_z}\right)^{-1}|dz|\\
		=&\int_{\partial D}|F_0|^2e^{-\varphi}c(-2\psi)\left|\bigtriangledown\psi\right|^{-1}|dz|\\
		\le&\liminf_{r\rightarrow1-0}\frac{\int_{\frac{1}{2}\log r}^0\left(\int_{\{z\in D:\psi(z)=s\}}|F_0|^2e^{-\varphi}c(-2\psi)\left|\bigtriangledown\psi\right|^{-1}|dz|\right)ds}{-\frac{1}{2}\log r}\\
		=&\liminf_{r\rightarrow1-0}\frac{\int_{\{z\in D:2\psi(z)\ge\log r\}}|F_0|^2e^{-\varphi}c(-2\psi)}{\int_0^{-\log r}c(t)e^{-t}dt}\times\frac{\int_0^{-\log r}c(t)e^{-t}dt}{-\frac{1}{2}\log r}\\
		=&2\liminf_{r\rightarrow1-0}\frac{\int_{\{z\in D:2\psi(z)\ge\log r\}}|F_0|^2e^{-\varphi}c(-2\psi)}{\int_0^{-\log r}c(t)e^{-t}dt}.
	\end{split}
\end{equation}
As $F_0\in H_2^{(c)}(D)$, we have $1=F_0(z_0)=\frac{1}{2\pi}\int_{\partial D}F_0(z)\overline{K_{\rho\left(\frac{\partial\psi}{\partial v_z}\right)^{-1}}(z,\overline{z_0})}\rho\left(\frac{\partial\psi}{\partial v_z}\right)^{-1}|dz|$. By Cauchy-Schwarz inequality, it follows that
\begin{equation}
	\label{eq:0709f}\begin{split}
		1&\le \frac{1}{(2\pi)^2}\left(\int_{\partial D}|F_0|^2\rho\left(\frac{\partial\psi}{\partial v_z}\right)^{-1}|dz|\right)\times\left(\int_{\partial D}|K_{\rho\left(\frac{\partial\psi}{\partial v_z}\right)^{-1}}(z,\overline{z_0})|^2\rho\left(\frac{\partial\psi}{\partial v_z}\right)^{-1}|dz|\right)\\
		&=\frac{1}{2\pi}\left(\int_{\partial D}|F_0|^2\rho\left(\frac{\partial\psi}{\partial v_z}\right)^{-1}|dz|\right)\times K_{\rho,\psi}(z_0).
	\end{split}
\end{equation}
Combining inequality \eqref{eq:0709d}, inequality \eqref{eq:0709e} and inequality \eqref{eq:0709f}, we obtain that
\begin{equation}
	\label{eq:0709g}
	\begin{split}
		\left(\int_0^{+\infty}c(t)e^{-t}dt\right)B_{\rho}(z_0)&=\frac{\int_0^{+\infty}c(t)e^{-t}dt}{G(0)}\\
		&\leq\limsup_{r\rightarrow 1-0}\frac{\int_0^{-\log r}c(t)e^{-t}dt}{\int_{\{z\in D:2\psi(z)\ge\log r\}}|F_0|^2e^{-\varphi}c(-2\psi)}\\
		&\le 2\left(\int_{\partial D}|F_0|^2e^{-\varphi}c(-2\psi)\left(\frac{\partial\psi}{\partial v_z}\right)^{-1}|dz|\right)^{-1}\\
		&\le \frac{1}{\pi} K_{\rho,\psi}(z_0).
	\end{split}
\end{equation}
Thus, we have proved the inequality part of Theorem \ref{main theorem}.

\

\emph{Step 2:} Assume that the equality
\begin{equation}
	\label{eq:0710c}K_{\rho,\psi}(z_0)=\left(\int_0^{+\infty}c(t)e^{-t}dt\right)\pi B_{\rho}(z_0)
\end{equation}
holds. Then inequality \eqref{eq:0709g} becomes an equality, which shows that 
\begin{equation*}
	\frac{\int_0^{+\infty}c(t)e^{-t}dt}{G(0)}=\limsup_{r\rightarrow 1-0}\frac{\int_0^{-\log r}c(t)e^{-t}dt}{\int_{\{z\in D:2\psi(z)\ge\log r\}}|F_0|^2e^{-\varphi}c(-2\psi)}.
\end{equation*}
Following from the concavity of $G(h^{-1}(r))$, we obtain that $G(h^{-1}(r))$ is linear with respect to $r\in(0,\int_0^{+\infty}c(t)e^{-t}dt)$. Theorem \ref{thm:m-points} shows that the following the following statements hold:
	
	$(1)$ $\psi=p_0G_{D}(\cdot,z_0)$, where $p_0=v(dd^c(\psi),z_0)>0$;
	
	$(2)$ $\varphi+2\psi=2\log|g|+2G_{D}(\cdot,z_0)+2u_1$, where $g$ is a holomorphic function on $D$ such that $ord_{z_0}(g)=0$ and $u_1$ is a harmonic function on $D$;
	
	$(3)$ $\chi_{z_0}=\chi_{-u_1}$.

In the following, we will prove that $2\log|g|$ is harmonic on $D$, a.e., $g(z)\not=0$ holds for any $z\in D$.
 
Denote that $h:=\varphi+2\psi-2G_{D}(\cdot,z_0)$ is a function on $\overline D$, thus $h$ is subharmonic on $D$ and $h$ is continuous at $z$ for any $z\in\partial D$.
 By the analyticity of $\partial D$, there exists $\tilde h\in C(\overline D)$ such that $\tilde h|_{\partial D}=h|_{\partial D}$ and $\tilde h$ is harmonic on $D$. As $h$ is subharmonic on $D$, we have 
 $$h\le \tilde h$$
  on $D$. Denote that
  $$\tilde\varphi:=\varphi+\tilde h-h.$$ Then we have $\tilde\varphi|_{\partial D}=\varphi|_{\partial D}$ and $\tilde \varphi+\psi=2G_D(\cdot,z_0)+\tilde h$. Denote that $\tilde\rho:=e^{-\tilde\varphi}c(-2\psi)$. It is clear that 
 \begin{displaymath}
 	K_{\tilde\rho,\psi}(z_0)=K_{\rho,\psi}(z_0)\,\,\text{and}\,\,B_{\tilde\rho}(z_0)\ge B_{\rho}(z_0).
 \end{displaymath}
Following equality \eqref{eq:0710c} and the result in Step 1, we have 
$$\frac{K_{\rho,\psi}(z_0)}{\int_0^{+\infty}c(t)e^{-t}dt}=\pi B_{\rho}(z_0)\le \pi B_{\tilde\rho}(z_0)\le \frac{K_{\tilde\rho,\psi}(z_0)}{\int_0^{+\infty}c(t)e^{-t}dt}=\frac{K_{\rho,\psi}(z_0)}{\int_0^{+\infty}c(t)e^{-t}dt},$$
which implies that 
$$B_{\rho}(z_0)=B_{\tilde\rho}(z_0).$$ Then we have $\tilde\rho=\rho$, a.e., $\tilde h=h$, which implies that $2\log|g|$ is harmonic on $D$. Denote that 
$$u=\log|g|+u_1$$
 is a harmonic function on $D$. Then we have $\varphi+2\psi=2G_{D}(\cdot,z_0)+2u$ and $\chi_{z_0}=\chi_{-u_1}=\chi_{-u}$.

\

\emph{Step 3:} Assume that the three statements $(1)-(3)$ hold.

It follows from Theorem \ref{thm:m-points} that $G(h^{-1}(r))$ is linear with respect to $r\in(0,\int_0^{+\infty}c(t)e^{-t}dt)$. By Corollary \ref{c:linear} and Remark \ref{rem:1.1}, we get that 
\begin{equation}
	\label{eq:220710a}
	G(t)=\int_{\{2\psi<-t\}}|F_0|^2e^{-\varphi}c(-2\psi)
	\end{equation}
holds for any $t\geq0$ and 
$$F_0=c_0(p_*(f_{z_0}))'p_*(f_u),$$
 where $c_0$ is a constant, $p$ is the universal covering from unit disc $\Delta$ to $ D$, $f_u$ is a holomorphic function on $\Delta$ such that $|f_u|=p^*(e^u)$, and $f_{z_0}$ is a holomorphic function on $\Delta$ such that $|f_{z_0}|=p^*(e^{G_D(\cdot,z_0)})$. It follows from equality \eqref{eq:220710a} that 
 \begin{equation}
 	\label{eq:220710f}\frac{\int_0^{+\infty}c(t)e^{-t}dt}{G(0)}=\limsup_{r\rightarrow 1-0}\frac{\int_0^{-\log r}c(t)e^{-t}dt}{\int_{\{z\in D:2\psi(z)\ge\log r\}}|F_0|^2e^{-\varphi}c(-2\psi)}.
 \end{equation}

As $u=\frac{\varphi}{2}+\psi-G_D(\cdot,z_0)$, we have $u\in C(\overline D)$, which implies that $p_*(|f_u|)\in C(\overline D)$. As $G_{D}(\cdot,z_0)$ can be extended to a harmonic function on a $U\backslash\{z_0\}$, where $U$ is a neighborhood of $\overline D$, we have $|(p_*(f_{z_0}))'|\in C(\overline D)$. Thus, we have 
$$|F_0|\in C(\overline D).$$ Following from the dominated convergence theorem and Lemma \ref{l:3}, we obtain that
\begin{equation}
	\label{eq:220710g}\limsup_{r\rightarrow 1-0}\frac{\int_{\{z\in D:2\psi(z)\ge\log r\}}|F_0|^2e^{-\varphi}c(-2\psi)}{\int_0^{-\log r}c(t)e^{-t}dt}=\frac{1}{2}\int_{\partial D}|F_0|^2e^{-\varphi}c(-2\psi)\left(\frac{\partial\psi}{\partial v_z}\right)^{-1}|dz|.
\end{equation}

Denote that $M(z):=\frac{K_{\lambda}(z,\overline{z_0})}{K_{\lambda}(z_0,\overline{z_0})}$, where $\lambda=\rho\left(\frac{\partial\psi}{\partial v_z} \right)^{-1}$. Note that $\int_D|F_0|^2e^{-\varphi}c(-2\psi)<+\infty$ implies that $\int_De^{-\varphi}c(-2\psi)<+\infty$. Lemma \ref{l:szego1} shows that $|M(z)|\in C(\overline D)$, then we have 
\begin{equation*}
	\int_{D}|M|^2e^{-\varphi}c(-2\psi)<+\infty.
\end{equation*}
Note that $M(z_0)=1$. By using Lemma \ref{l:unique} and inequality \eqref{eq:220710a}, we have 
\begin{displaymath}\begin{split}
	\int_{\{2\psi<-t\}}|M|^2e^{-\varphi}c(-2\psi)=&\int_{\{2\psi<-t\}}|F_0|^2e^{-\varphi}c(-2\psi)\\
	&+\int_{\{2\psi<-t\}}|M-F_0|^2e^{-\varphi}c(-2\psi),
\end{split}
\end{displaymath}
which implies that 
\begin{equation}
	\label{eq:220710d}\int_{\{2\psi<-t\}}F_0\overline{F_0-M} e^{-\varphi}c(-2\psi)=0
\end{equation}
holds for any $t\ge0$. Note that $\psi=p_0G_D(\cdot,z_0)$. It follows from Lemma \ref{l:3} and equality \eqref{eq:220710d} that there exists $r_1>0$ such that 
\begin{equation}
	\label{eq:220710c}\int_{\{z\in D:G_D(z,z_0)=r\}}F_0\overline{F_0-M} e^{-\varphi}\left(\frac{\partial G_D(\cdot,z_0)}{\partial v_z} \right)^{-1}|dz|=0
\end{equation}
holds for any $r\in(0,r_1)$. Note that $|F_0|\in C(\overline D)$ and $|M|\in C(\overline D)$, then it follows from the dominated convergence theorem and equality \eqref{eq:220710c} that 
\begin{displaymath}
	\int_{\partial D}F_0\overline{F_0-M} e^{-\varphi}\left(\frac{\partial G_D(\cdot,z_0)}{\partial v_z} \right)^{-1}|dz|=0,
\end{displaymath}
which implies that 
\begin{equation}\nonumber
\begin{split}
		&\int_{\partial D}|M|^2 e^{-\varphi}\left(\frac{\partial G_D(\cdot,z_0)}{\partial v_z} \right)^{-1}|dz|\\
		=&\int_{\partial D}|M-F_0|^2 e^{-\varphi}\left(\frac{\partial G_D(\cdot,z_0)}{\partial v_z} \right)^{-1}|dz|+\int_{\partial D}|F_0|^2 e^{-\varphi}\left(\frac{\partial G_D(\cdot,z_0)}{\partial v_z} \right)^{-1}|dz|.
	\end{split}
\end{equation}
Lemma \ref{l:szego2} tells us that 
$$\int_{\partial D}|M|^2 e^{-\varphi}\left(\frac{\partial G_D(\cdot,z_0)}{\partial v_z} \right)^{-1}|dz|\le \int_{\partial D}|F_0|^2 e^{-\varphi}\left(\frac{\partial G_D(\cdot,z_0)}{\partial v_z} \right)^{-1}|dz|.$$
Then we have 
$$\int_{\partial D}|M|^2 e^{-\varphi}\left(\frac{\partial G_D(\cdot,z_0)}{\partial v_z} \right)^{-1}|dz|= \int_{\partial D}|F_0|^2 e^{-\varphi}\left(\frac{\partial G_D(\cdot,z_0)}{\partial v_z} \right)^{-1}|dz|.$$
It follows from Lemma \ref{l:szego2} that 
\begin{equation}
	\label{eq:0711c}F_0\equiv M.
\end{equation}
 Thus, inequality \eqref{eq:0709f} becomes equality, i.e.
\begin{equation}
	\label{eq:220710e}1=\frac{1}{2\pi}\left(\int_{\partial D}|F_0|^2\rho\left(\frac{\partial\psi}{\partial v_z}\right)^{-1}|dz|\right)\times K_{\rho,\psi}(z_0).
\end{equation}
Combining equality \eqref{eq:220710f}, equality \eqref{eq:220710g} and equality \eqref{eq:220710e}, we konw that  inequality \eqref{eq:0709g} becomes equality, i.e.
$$\left(\int_0^{+\infty}c(t)e^{-t}dt\right)B_{\rho}(z_0)= \frac{1}{\pi} K_{\rho,\psi}(z_0).$$

Then Theorem \ref{main theorem} has been proved.
\end{proof}

\begin{proof}
	[Proof of Remark \ref{rem:function}]Assume that the three statements $(1)-(3)$ in Theorem \ref{main theorem} hold. Following the discussions in Step 3 in the proof of Theorem \ref{main theorem}, we obtain that 
	\begin{displaymath}
		F_0=c_0(p_*(f_{z_0}))'p_*(f_u),\,\,F_0\equiv M\,\,\text{and}\,\,M(z)=\frac{K_{\lambda}(z,\overline{z_0})}{K_{\lambda}(z_0,\overline{z_0})},
	\end{displaymath} where $\lambda=\rho\left(\frac{\partial\psi}{\partial v_z} \right)^{-1}$.	
	Thus, we have 
	$$K_{\rho,\psi}(\cdot,\overline{z_0})=K_{\rho,\psi}(z_0,\overline{z_0})F_0=c_1(p_*(f_{z_0}))'p_*(f_u),$$ where $c_1$ is a constant. As 
	$$\int_D\left|\frac{B_{\rho}(\cdot,\overline{z_0})}{B_{\rho}(z_0,\overline{z_0})}\right|^2\rho=\frac{1}{B_{\rho}(z_0,\overline{z_0})}=G(0),$$ it follows from Lemma \ref{l:unique} that 
	$$\frac{B_{\rho}(\cdot,\overline{z_0})}{B_{\rho}(z_0,\overline{z_0})}=F_0.$$ Theorem \ref{main theorem} shows that 
	$K_{\rho,\psi}(z_0,\overline{z_0})=\left(\int_0^{+\infty}c(t)e^{-t}dt\right)\pi B_{\rho}(z_0,\overline{z_0}),$
	thus we obtain
	$$K_{\rho,\psi}(\cdot,\overline{z_0})=\left(\int_0^{+\infty}c(t)e^{-t}dt\right)\pi B_{\rho}(\cdot,\overline{z_0}).$$
	\end{proof}

\section{Proof of Corollary \ref{c:higher}}

In this section, we prove Corollary \ref{c:higher} by using Theorem \ref{main theorem}.

Let $\tilde\varphi=\varphi-2k\log|z-z_0|$, then it is clear that $\tilde\varphi+2\psi$ is subharmonic on $D$ and $v(dd^c(\tilde\varphi+2\psi),z_0)\ge2$. Denote that $\tilde\rho:=e^{-\tilde\varphi}c(-2\psi)=|z-z_0|^{2k}\rho$. 
Note that 
\begin{displaymath}
	\begin{split}
		&B^{(k)}_{\rho}(z_0)\\
		=&\sup\left\{\left|\frac{f^{(k)}(z_0)}{k!}\right|^2:f\in\mathcal{O}(D),\,\,\int_D|f|^2\rho\le1\,\,\&\,\,f(z_0)=\ldots=f^{(k-1)}(z_0)=0 \right\}\\
		=&\sup\left\{|g(z_0)|^2:g\in\mathcal{O}(D)\,\,\&\,\,\int_D|g|^2\tilde\rho\le1\right\}\\
		=&B_{\tilde\rho}(z_0),
	\end{split}
\end{displaymath}
and
\begin{displaymath}
	\begin{split}
		&K_{\rho,\psi}^{(k)}(z_0)\\
		=&\sup\Bigg\{\left|\frac{f^{(k)}(z_0)}{k!}\right|^2:f\in H^{(c)}_2(D),\,\, \int_{\partial D}|f|^2\rho\left(\frac{\partial\psi}{\partial v_z} \right)^{-1}|dz|\le1\\
		&\&\,\, f(z_0)=\ldots=f^{(k-1)}(z_0)=0 \Bigg\}\\
		=&\sup\Bigg\{|g(z_0)|^2:g\in H^{(c)}_2(D)\,\,\&\,\,\int_{\partial D}|g|^2\tilde\rho\left(\frac{\partial\psi}{\partial v_z} \right)^{-1}|dz|\le1\Bigg\}\\
		=&K_{\tilde\rho,\psi}(z_0).
	\end{split}
\end{displaymath}
Theorem \ref{main theorem} tell us that 
\begin{equation}
	\label{eq:0711a}K_{\tilde\rho,\psi}(z_0)\ge \left(\int_0^{+\infty}c(t)e^{-t}dt\right)\pi B_{\tilde\rho}(z_0)
\end{equation}
holds and the equality holds if and only if the following statements holds:
	
\noindent$(1)$ $\tilde\varphi+2\psi=2G_{D}(\cdot,z_0)+2u_1$, where $u_1$ is a harmonic function on $D$;
	 
\noindent	$(2)$ $\psi=p_0G_{D}(\cdot,z_0)$, where $p_0=v(dd^c(\psi),z_0)>0$;
	
\noindent	$(3)$ $\chi_{z_0}=\chi_{-u_1}$.

\noindent Then inequality \eqref{eq:0711a} implies that 
\begin{equation}
	\label{eq:0711b}K_{\rho,\psi}^{(k)}(z_0)\ge \left(\int_0^{+\infty}c(t)e^{-t}dt\right)\pi B^{(k)}_{\rho}(z_0)
\end{equation}
holds. Let $u(z)=u_1(z)+k(\log|z-z_0|-G_D(z,z_0))$ on $D$, then it follows from Lemma \ref{l:green-sup} that $u$ is harmonic on $D$ if and only if $u_1$ is harmonic on $D$.  It is clear that $\chi_{-u}\chi_{z_0}^k=\chi_{-u_1}$ when $u$ is harmonic on $D$. Thus, the equality in \eqref{eq:0711b} holds if and only if the following statements holds:
	
\noindent$(1)$ $\varphi+2\psi=2(k+1)G_{D}(\cdot,z_0)+2u$, where $u$ is a harmonic function on $D$;
	 
\noindent	$(2)$ $\psi=p_0G_{D}(\cdot,z_0)$, where $p_0=v(dd^c(\psi),z_0)>0$;
	
\noindent	$(3)$ $\chi_{z_0}^{k+1}=\chi_{-u}$.

Thus, Corollary \ref{c:higher} holds.

\vspace{.1in} {\em Acknowledgements}. The authors would like to thank Dr. Shijie Bao, Dr. Zhitong Mi and Gan Li for checking the manuscript and  pointing out some typos. The first named author was supported by National Key R\&D Program of China 2021YFA1003103, NSFC-11825101, NSFC-11522101 and NSFC-11431013.

\bibliographystyle{references}
\bibliography{xbib}

\end{document}